\documentclass[reqno,12pt]{amsart}
\usepackage{float,epsfig,here}
\usepackage{enumerate}
\usepackage{amsmath,amssymb}
\newtheorem{remark}{Remark}

\newtheorem{assumption}{Assumption}

\newtheorem{theorem}{Theorem}

\allowdisplaybreaks
\newtheorem{definition}{Definition}

\begin{document}

\baselineskip=15pt

\title[]
{Partially Observable Discrete-time Discounted Markov Games with General Utility }
\thanks{}
\author[Bhabak]{Arnab Bhabak}
\address{Department of Mathematics\\
Indian Institute of Technology Guwahati\\
Guwahati, Assam, India}
\email{bhabak@iitg.ac.in}

%\author[Pal]{Chandan Pal}
%\address{Department of Mathematics\\
%Indian Institute of Technology Guwahati\\
%Guwahati, Assam, India}
%\email{cpal@iitg.ac.in}

\author[Saha]{Subhamay Saha}
\address{Department of Mathematics\\
Indian Institute of Technology Guwahati\\
Guwahati, Assam, India}
\email{saha.subhamay@iitg.ac.in}

%---------------------------------------------------------------------------

\date{}

\begin{abstract} In this paper, we investigate a partially observable zero sum games where the state process is a discrete time Markov chain. We consider a general utility function in the optimization criterion. We show the existence of value for both finite and infinite horizon games and also establish the existence of optimal polices. The main step involves converting the partially observable game into a completely observable game which also keeps track of the total discounted accumulated reward/cost.
\vspace{2mm}


\end{abstract}

\maketitle

	\section{Introduction}
	\noindent In this paper we consider a discrete time zero-sum game where the state process evolves like a controlled Markov chain. We also assume that the state has two components one of which is observable while the other is not observable. In the optimization criterion we consider general utility function and discounted reward/cost.  Player 1 is assumed to be the maximizer and player 2 is the minimizer. Both finite and infinite horizon problems are investigated. In both cases we show that the game has value and also establish the existence of optimal policies for both players.
	   
	   Risk-sensitive optimization problems wherein one tries to optimize the expectation of the exponential of the accumulated cost/reward, are widely studied in stochastic optimal control literature. One of the primary reasons for the popularity of this criterion is that unlike risk-neutral optimization criterion, here the risk preference of the controllers are taken into account. Risk-sensitive control problems have been widely studied in literature for both Markov as well as semi-Markov processes. Risk-sensitive control problems for discrete-time Markov chains has been studied in \cite{Rieder14, Meyn02, Sobel87, Stettner07, Marcus96, Jask07}, for continuous-time Markov chains in \cite{Saha14, Wei19, Zhang19, Pal13, Chen19} and for semi-Markov processes in \cite{Saha22, Cadena16, Lian18}. There is also a good amout of literature on risk-sensitive games, both zero-sum and non-zero sum, see \cite{Ghosh14, Rieder17b, Saha21, Daniel19, Pal16, Wei18}. But, in all the above cited literature it is assumed that the state process is completely observable. However, in practice it may be the case that complete information about the state is unavailable to the controller for taking decisions, which makes the study of optimization problems with partial observation quite natural. Risk-sensitive control problems for discrete-time Markov chains with partial observation has been studied in \cite{Rieder17a, Stettner99, Elliott94}. Although, there has been work on partially observable risk-neutral games \cite{Sinha04, Goswami06, Goswami08, Saha14a}, but to the best of our knowledge there has been no work on risk-sensitive games with partial observation. Thus, this work seems to be the first work on partially observable risk-sensitive games.
	   In this paper we consider risk-sensitive zero-sum games with partial observation and general utility function, and thus the exponential utility is a special case of our model. Like in the risk-neutral case here also we convert the partially observable game into an equivalent completely observable game. But the difference with the risk-neutral case is that here we need to keep track of the accumulated cost as well. Since in the considered model the reward/cost function is assumed to be dependant on the unobservable component, so we need to consider the joint conditional distribution of the unobservable state component and the accumulated reward/cost.  
	   
	   The rest of the paper is organized a follows. In Section 2, we describe our game model. Section 3 deals with finite horizon problem. Finally in Section 4 we investigate the infinite horizon problem as a limit of finite horizon problem.
	   \section{Zero-Sum Game Model}

The partially observable risk-sensitive zero-sum Markov game model that we are interested in can be represented by the tuple 
\begin{align}\label{Game model}
(X,Y,A,B,\{A(x)\subset A, B(x)\subset B, x\in X\}, C(x,y,a,b),Q),
\end{align}
where each individual component has the following interpretation. \begin{itemize}
	\item $X,Y$ are the Borel spaces, X represents the observable state space and Y is the non-observable space.
	\item The Borel spaces $A$ and $B$ are the action sets for player 1 and 2 respectively. And for each $x\in X$, $A(x)\subset A$, $B(x)\subset B$ are Borel subsets denoting the set of all admissible actions in state $x\in X$ for player 1 and 2 respectively.
	\item Define $\mathbb{K}=\{(x,a,b): x\in X, a\in A(x), b\in B(x)\}$ to be the set of admissible state-action pairs, which is assumed to be a measurable subset of $X\times A\times B$. Then $C:\mathbb{K}\times Y\to \mathbb{R}$  is the immediate reward function for player 1 and immediate cost function for player 2 respectively. 
	\item For each $(x,y,a,b)\in \mathbb{K} \times Y$, the mapping $Q(B\vert x,y,a,b)$ is the transition probability that the next pair is in $B\in \mathcal{B}(X\times Y)$, where $\mathcal{B}(X\times Y)$ is the collection of all Borel subsets of $X\times Y$,  given the current state is $(x,y)$ and actions $(a,b)\in A(x)\times B(x)$ are chosen by the players.
\end{itemize}
Also we have the discount factor $\beta\in(0,1)$. 
In what follows, we assume that the transition kernel $Q$ has a measurable density $q$ with respect to some $\sigma$-finite measures $\lambda$ and $\nu$, i.e.,
\begin{align}
Q(B\vert x,y,a,b)=\int_{B} q(x^{'},y^{'}\vert x,y,a,b)\lambda(dx^{'})\nu(dy^{'}), \hspace{.5cm} B\in \mathcal{B}(X\times Y).
\end{align}
We also introduce the marginal transition kernel density by
\begin{align*}
q^{X}(x^{'}\vert x,y,a,b):=\int_{Y}q(x^{'},y^{'}\vert x,y,a,b)\nu(dy^{'}).
\end{align*}
We assume that the distribution $Q_{0}$ of $Y_{0}$, the initial (unobservable) state is known to the players. The risk-sensitive partially observed zero sum game evolves in the following way: 
\begin{itemize}
	\item At the $0$th decision epoch, based on the initial observation $x_{0}$, the players choose actions $a_0\in A(x_0)$ and $b_{0}\in B(x_0)$ simultaneously and independent of each other.
    \item As a consequence of these chosen actions player 1 gets a reward $C(x_{0},y_{0},a_0,b_0)$ and player 2 incurs a cost $C(x_{0},y_{0},a_0,b_0)$, where $y_0$ is the initial unobservable state. Note that the reward/cost depends on the unobservable component as well and therefore is itself unobservable.
    \item At the next time epoch the system moves to the next state $(x_1,y_1)$ according to the transition law $Q(.\vert x_{0},y_{0},a_0,b_0)$. If the observable state at time 1 is $x_1$, then based on this observation, the previous observation $x_0$ and the previous pair of actions $(a_0,b_0)$, players 1 and 2 choose actions $a_1\in A(x_1)$ and $b_1\in B(x_1)$ respectively. This yields a reward  $C(x_1,y_1,a_1,b_1)$ for player 1 and a cost $C(x_1,y_1,a_1,b_1)$ for player 2. Now the sequence of events as described above repeats itself.
\end{itemize}
Let $\mathcal{H}_{n}$ be the admissible observable history available upto time n, i.e., $\mathcal{H}_{0}=X$ and  for $n\geq 1$, $\mathcal{H}_{n}=\mathcal{H}_{n-1}\times A \times B \times X$. Thus a typical element of $\mathcal{H}_n$ is given by $h_n=(x_0,a_0,b_0,x_1,a_1,b_1,\ldots,x_{n-1},a_{n-1},b_{n-1},x_n)$, where for each $n$, $x_n$ represents the observable state at the $nth$ decision epoch, $a_n$ is the action chosen by player 1 at the $nth$ decision epoch and $b_n$ is the action chosen by player 2 at the $nth$ decision epoch. We endow these spaces with the Borel sigma-algebra. 
Next we introduce the concept of decision rules and policies.

\begin{definition} Let $P(A)$ and $P(B)$ denote the set of all probability measures on $A$ and $B$ respectively.
	\item[(a)] A measurable mapping $f_{n}:\mathcal{H}_{n}\rightarrow P(A)$ with the property $f_{n}(h_{n})(A(x_n))=1$ for $h_{n}\in \mathcal{H}_{n}$ is called a decision rule at stage n for player 1. Similarly,  a measurable mapping $g_{n}:\mathcal{H}_{n}\rightarrow P(B)$ with the property $g_{n}(h_{n})(B(x_n))=1$ for $h_{n}\in \mathcal{H}_{n}$ is called a decision rule at stage n for player 2.
	\item [(b)] A sequence $\pi=(f_{0},f_{1},...)$, and $\sigma=(g_{0},g_{1},...)$, where $f_{n},g_{n}$'s are decision rules at stage $n$ for all $n$, is called a pair of policies for player 1 and 2 respectively.
\end{definition}

\section{Finite Horizon Problem}

For a fixed policy $\pi=(f_{0},f_{1},...)$ and $\sigma=(g_{0},g_{1},...)$, fixed (observable) initial state $x\in X$, the initial distribution $Q_{0}$ together with the transition kernel $Q$, we obtain by a theorem of Ionescu Tulcea a probability measure $\mathbb{P}^{\pi \sigma}_{xy}$ on $(X\times Y)^{\infty}$ endowed with the product $\sigma$-algebra. More precisely $\mathbb{P}^{\pi \sigma}_{xy}$ is the probability measure under policy $(\pi,\sigma)$ given $X_{0}=x$ and $Y_{0}=y$. On this probability space, for $\omega=(x_0,y_0,x_1,y_1,\ldots,x_n,y_n,\ldots)$ we define the random variables $X_n$ and $Y_n$ via the canonical projections:
$$X_n(\omega)=x_n,\quad Y_n(\omega)=y_n.$$ The action sequence for player 1 is given by $\{A_n\}$ and that of player 2 given by $\{B_n\}$.
Under the policies, $\pi=(f_{0},f_{1},...)$ and $\sigma=(g_{0},g_{1},...)$, the distribution of $A_n$ is given by $f_n(X_0,A_0,B_0,\ldots,X_n)$ and the distribution of $B_n$ is given by $g_n(X_0,A_0,B_0,\ldots,X_n)$.

In this section we look into the $N$ stage optimization problem. For defining the optimality criterion, consider a utility function $U:\mathbb{R}_{+}\rightarrow \mathbb{R}$ which is assumed to be continuous and strictly increasing. 
The discounted cost/reward generated over N stages is given by:
\begin{align}
C_{N}=\sum_{k=0}^{N-1}\beta^{k}C(X_{k},Y_{k},A_{k},B_{k}).
\end{align}
For a fixed initial observable state $x$ and given policies $\pi$ and $\sigma$, the optimization criterion that we are interested in is given by:
\begin{align}\label{optimization problem}
J_{N \pi \sigma}(x):=\int_{Y}\mathbb{E}^{\pi \sigma}_{xy}[U( C_{N})]Q_{0}(dy).
\end{align}
Here player 1 tries to maximize $J_{N \pi \sigma}(x)$ over all policies $\pi$, for each $\sigma$. Analogously, player 2 tries to minimize $J_{N \pi \sigma}(x)$ over all policies $\sigma$, for each $\pi$. This leads to the following definitions of optimal policies and value of the game.
\begin{definition}\label{optimal}
A strategy $\pi^{*}$ is said to be optimal for player 1 in the partially observed model if\\
\begin{align*}
J_{N\pi^{*}\sigma}(x)\geq \sup_{\pi} \inf_{\sigma^{'}} J_{N\pi\sigma^{'}}(x), \hspace{.3cm} for\hspace{.1cm} any \hspace{.2cm} \sigma. 
\end{align*} 
The quantity $\sup_{\pi} \inf_{\sigma^{'}} J_{N\pi\sigma^{'}}(x,y,z)$ is referred to as the lower value of the partially observed game.\\
Similarly, a strategy $\sigma^{*}$ is said to be optimal for player 2 in the partially observed model if\\
\begin{align*}
J_{N\pi\sigma^{*}}(x)\leq \inf_{\sigma}\sup_{\pi^{'}}  J_{N\pi^{'}\sigma}(x), \hspace{.3cm} for\hspace{.1cm} any \hspace{.2cm} \pi.
\end{align*} 
The quantity $\inf_{\sigma}\sup_{\pi^{'}}  J_{N\pi^{'}\sigma}(x,y,z)$ is referred as the upper value of the partially observed game.\\
Hence, a pair optimal strategies $(\pi^{*},\sigma^{*})$ satisfies
\begin{align*}
J_{N\pi\sigma^{*}}\leq J_{N\pi^{*}\sigma^{*}} \leq J_{N\pi^{*}\sigma}
\end{align*} 
for any $\pi ,\sigma$. Thus, $(\pi^{*},\sigma^{*})$ constitutes a saddle point equilibrium. The partially observed game is said to have a value if
\begin{align*}
J_{N}(x)=\sup_{\pi} \inf_{\sigma} J_{N\pi\sigma}(x)=\inf_{\sigma}\sup_{\pi}  J_{N\pi\sigma}(x).
\end{align*}
\end{definition}
Note that if both the players have optimal strategies then the partially observed game has a value. Our aim is to show that the game model under consideration has a value and also there exists a saddle point equilibrium. Towards that end,  we assume that the following assumptions are in force throughout the paper.

\begin{assumption}$\mbox{}$\\

		\begin{itemize}
		\item[(i)] For each $x\in X$, the sets $A(x)$ and $B(x)$ are compact subsets of $A$ and $B$. Also the mappings $x\rightarrow A(x)$ and $x\rightarrow B(x)$ are continuous.
		\item [(ii)] $(x,y,a,b)\rightarrow C(x,y,a,b)$ is continuous. 
		\item[(iii)]  $(x,y,x^{'},y^{'},a,b)\rightarrow q(x^{'},y^{'}\vert x,y,a,b)$ is bounded and continuous. 
		\item [(iv)] C is also bounded, i.e., there exist constants $\underline{a}, \bar{a}$ with $0<\underline{a}\leq C(x,y,a,b)\leq
		\bar{a}$.
	\end{itemize}
\end{assumption}
In literature partially observable risk-neutral games are treated by first converting it to an equivalent completely observable game. Following that approach, in this risk-sensitive approach also we first convert our partially observed game problem into a complete observed game problem. We show the existence of the value function and optimal strategies in case of the equivalent completely observable model and then revert back to our partially observed model.

Now in the unobserved model, the state $y$ and the cost accumulated so far cannot be observed because it depends on $y$. Thus we need to estimate them. For that we consider the following set of probability measures on $Y\times \mathbb{R}_{+}$:\\
\begin{center}
	$P_b(Y\times \mathbb{R}_{+})$:=\{$\mu$ is a probability measure on the Borel $\sigma$- algebra $\mathcal{B}(Y\times \mathbb{R}_{+})$ such that there exists a constant $K=K(\mu)>0$ with $\mu(Y\times [0,K])=1$\}.
\end{center}
The elements of the above set will essay the role of the conditional distribution  of the hidden state component and accumulated cost. The precise interpretation will be seen in Theorem \ref{Connection between existing probability}. In order to estimate the unobserved state component and accumulated cost we define the following updating operator $\Phi:X\times A \times B\times X\times P_b(Y\times \mathbb{R}_{+})\times \mathbb{R}_{+}\rightarrow P_b(Y\times \mathbb{R}_{+})$ given by
\begin{align}\label{updating}
\Phi(x,a,b,x^{'},\mu, z)(B):=\frac{\int_{Y}\int_{\mathbb{R}_{+}}\big(\int_{B}q(x^{'},y^{'}\vert x,y,a,b)\nu(dy^{'})\delta_{s+zC(x,y,a,b)}(ds^{'})\big)\mu(dy,ds)}{\int_{Y}q^{X}(x^{'}\vert x,y,a,b)\mu^{Y}(dy)},
\end{align}
where $B\in \mathcal{B}(Y\times \mathbb{R}_{+})$ and $\mu^{Y}(dy)=\mu(dy, \mathbb{R}_{+})$ is the Y-marginal distribution of $\mu$. Going forward we will also use the notation $\mu^{s}(ds):=\mu(Y,ds)$, which will denote the S-marginal. We define the updating operator only when the denominator is positive. For $h_{n}=(x_{0},a_{0},b_{0},.....,x_{n})$ and  $B\in \mathcal{B}(Y\times \mathbb{R}_{+})$ we now define the sequence of probability measures
\begin{align}\label{prob measure}
&\mu_{0}(B\vert h_{0}):=(Q_{0}\times \delta_{0})(B),\nonumber\\
&\mu_{n+1}(B\vert h_{n},a,b,x^{'})= \Phi(x_{n},a,b,x^{'},\mu_{n}(.\vert h_{n}), \beta^{n})(B).
\end{align}
The next theorem provides the interpretation of the above defined sequence of probability measures $(\mu_{n})$ as the sequence of conditional distributions. For that purpose we first define the sequence of random variables:

\begin{align*}
S_{0}:=0, \hspace{.2cm} S_{n}:=\sum_{k=0}^{n-1}\beta^{k}C(X_{k},Y_{k},A_{k},B_{k}),   \hspace{.2cm} n\in \mathbb{N}
\end{align*}
We then have the following result which generalizes Theorem 1 of \cite{Rieder17a} to the game setting.
\begin{theorem}\label{Connection between existing probability}
	Suppose that $(\mu_{n})$ is given by the recursion (\ref{prob measure}). For $n\geq 0$, for a given initial observable state $x\in X$ and given policies $\pi=(f_0,f_1,\ldots)$ and $\sigma=(g_0,g_1,\ldots)$ of the respective players we have
	{\small\begin{align*}
	\mathbb{P}_{x}^{\pi \sigma}((Y_{n},S_{n})\in B\vert X_{0},A_{0},B_{0},....,X_{n})=\mu_{n}(B\vert X_{0},A_{0},B_{0},....,X_{n}), \hspace{.2cm} \mathbb{P}^{\pi \sigma}_{x}-a.s., \hspace{.2cm} for \hspace{.2cm} B\in \mathcal{B}(Y\times \mathbb{R}_{+}),
	\end{align*}} where $\mathbb{P}_{x}^{\pi \sigma}(.):=\int \mathbb{P}^{\pi \sigma}_{xy}(.)Q_{0}(dy)$.
\end{theorem}

\begin{proof} We first show that 
\begin{align}\label{expectation}
\mathbb{E}^{\pi \sigma}_{x}[v(X_{0},A_{0},B_{0},X_{1},....,X_{n},Y_{n},S_{n})]=\mathbb{E}^{\pi \sigma}_{x}[v^{'}(X_{0},A_{0},B_{0},X_{1},....,X_{n})]
\end{align}
for all bounded and measurable $v:\mathcal{H}_{n}\times Y\times \mathbb{R}_{+}\rightarrow \mathbb{R}$ and 
\begin{align*}
v^{'}(h_{n}):=\int_{Y}\int_{\mathbb{R}_{+}}v(h_{n},y_{n},s_{n})\mu_{n}(dy_{n},ds_{n}\vert h_{n}).
\end{align*}
We use induction on $n$. The basis step is true, as for $n=0$, both sides reduce to $\int v(x,y,0)Q_{0}(dy)$. Now suppose that the statement is true for $n-1$. we simply write $f_{n},g_{n}$ in place of $f_{n}(h_{n}), g_{n}(h_{n})$. For a given observable history $h_{n-1}$, the left hand side of (\ref{expectation}) becomes:

\begin{align*}
&\mathbb{E}^{\pi \sigma}_{x}[v(h_{n-1},A_{n-1},B_{n-1},X_{n},Y_{n},S_{n})]=\int_B\int_A\int_{Y}\int_{\mathbb{R}_{+}}\mu_{n-1}(dy_{n-1},ds_{n-1}\vert h_{n-1})\times \\&
\int_{Y}\int_{X} \nu(dy_{n})\lambda(dx_{n})q(x_{n},y_{n}\vert x_{n-1},y_{n-1},a_{n-1},b_{n-1})\times\\&
\int_{\mathbb{R}_{+}} \delta_{s_{n-1}+\beta^{n-1}C(x_{n-1},y_{n-1},a_{n-1},b_{n-1})}(ds_{n})v(h_{n-1},a_{n-1},b_{n-1},x_{n},y_{n},s_{n})f_{n-1}(da_{n-1})g_{n-1}(db_{n-1})\\&
=\int_B\int_A\int_{Y}\int_{\mathbb{R}_{+}}\mu_{n-1}(dy_{n-1},ds_{n-1}\vert h_{n-1})
\int_{Y}\int_{X} \nu(dy_{n})\lambda(dx_{n})q(x_{n},y_{n}\vert x_{n-1},y_{n-1},a_{n-1},b_{n-1})\times\\&
 v(h_{n-1},a_{n-1},b_{n-1},x_{n},y_{n},s_{n-1}+\beta^{n-1}C(x_{n-1},y_{n-1},a_{n-1},b_{n-1}))f_{n-1}(da_{n-1})g_{n-1}(db_{n-1}).
\end{align*}
While the right hand side becomes:
\begin{align*}
&\mathbb{E}^{\pi \sigma}_{x}[v^{'}(X_{0},A_{0},B_{0},X_{1},....,X_{n})]\\&=\int_B\int_A\int_{Y}\int_{\mathbb{R}_{+}}\mu_{n-1}(dy_{n-1},ds_{n-1}\vert h_{n-1})\times \\&
\int_{X} \lambda(dx_{n})q^{X}(x_{n}\vert x_{n-1},y_{n-1},a_{n-1},b_{n-1})v^{'}(h_{n-1},a_{n-1},b_{n-1},x_{n})f_{n-1}(da_{n-1})g_{n-1}(db_{n-1})\\
&=\int_B\int_A\int_{Y}\mu_{n-1}^{Y}(dy_{n-1}\vert h_{n-1})\times 
\int_{X} \lambda(dx_{n})q^{X}(x_{n}\vert x_{n-1},y_{n-1},a_{n-1},b_{n-1})\times \\&
\int_{Y}\int_{\mathbb{R}_{+}} \mu_{n}(dy_{n},ds_{n}\vert h_{n})v(h_{n-1},a_{n-1},b_{n-1},x_{n},y_{n},s_{n})f_{n-1}(da_{n-1})g_{n-1}(db_{n-1})\\&
=\int_B\int_A\int_{Y}\int_{X}\nu(dy_{n})\lambda(dx_{n})\int_{Y}\int_{\mathbb{R}_{+}}\mu_{n-1}(dy_{n-1},ds_{n-1}\vert h_{n-1})q(x_{n},y_{n}\vert x_{n-1},y_{n-1},a_{n-1},g_{n-1})\times\\&
\int_{\mathbb{R}_{+}}\delta_{s_{n-1}+\beta^{n-1}C(x_{n-1},y_{n-1},f_{n-1},g_{n-1})}(ds_{n})v(h_{n-1},f_{n-1},g_{n-1},x_{n},y_{n},s_{n})f_{n-1}(da_{n-1})g_{n-1}(db_{n-1})\\&
=\int_B\int_A\int_{Y}\int_{\mathbb{R}_{+}}\mu_{n-1}(dy_{n-1},ds_{n-1}\vert h_{n-1})
\int_{Y}\int_{X} \nu(dy_{n})\lambda(dx_{n})q(x_{n},y_{n}\vert x_{n-1},y_{n-1},a_{n-1},b_{n-1})\times\\&
v(h_{n-1},a_{n-1},b_{n-1},x_{n},y_{n},s_{n-1}+\beta^{n-1}C(x_{n-1},y_{n-1},a_{n-1},b_{n-1}))f_{n-1}(da_{n-1})g_{n-1}(db_{n-1}),
\end{align*}	
where we use the recursion for $\mu_{n}$ in the third equation and use Fubini's theorem, to cancel out the normalizing constant of $\mu_{n}$. Hence we are done by induction.\\
Now, in particular, if we take $v=1_{B\times C}$ with $B\in \mathcal{B}(Y\times \mathbb{R}_{+})$ and $C\subset X\times A\times B\times ....\times X$ a measurable set of histories until time $n$ then we get from \eqref{expectation},
{\footnotesize\begin{align*}
\mathbb{P}_{x}^{\pi \sigma}((Y_{n},S_{n})\in B, (X_{0},A_{0},B_{0},....,X_{n})\in C)=\mathbb{E}^{\pi \sigma}_{x}[\mu_{n}(B\vert X_{0},A_{0},B_{0},....,X_{n})1_{C}(X_{0},A_{0},B_{0},....,X_{n})]
\end{align*} }
This establishes the fact that $\mu_{n}(B\vert X_{0},A_{0},B_{0},....,X_{n})$ is a conditional $\mathbb{P}^{\pi \sigma}_{x}$-distribution of $(Y_{n},S_{n})$ given the history $(X_{0},A_{0},B_{0},....,X_{n})$.
\end{proof}	
Now we again look at the optimization problem (\ref{optimization problem}). Motivated by the previous result we define for $x\in X$, $\mu\in P_b(Y\times \mathbb{R}_{+})$, $z\in (0,1]$ and $n=1,2,....,N$:

\begin{align}\label{complete}
V_{n\pi \sigma}(x,\mu,z):=\int_{Y} \int_{\mathbb{R}_{+}}\mathbb{E}^{\pi \sigma}_{xy}\bigg[U\big(s+z\sum_{k=0}^{n-1}\beta^{k}C(X_{k},Y_{k},A_{k},B_{k})\big)\bigg]\mu(dy,ds).
\end{align}	
 We solve our optimization problem by using a state augmentation technique. For that purpose we define, for a probability measure $\mu\in P(Y)$:
\begin{align*}
Q^{X}(B\vert x,\mu,a,b):=\int_{B}\int_{Y}q^{X}(x^{'}\vert x,y,a,b)\mu(dy)\lambda(dx^{'}),  \hspace{.5cm} B\in \mathcal{B}(X).
\end{align*}
	
We now consider the completely observable model with new state space $E=X\times P_b(Y\times \mathbb{R}_{+})\times (0,1]$. The action spaces for player 1 and player 2 are same as the partially observable model. One stage cost/reward is $0$ and the terminal cost/reward function is $V_{0}(x,\mu,z):=\int_{Y}\int_{\mathbb{R}_{+}}U(s)\mu(dy,ds)$. Since for all $\mu\in P_b(Y\times \mathbb{R}_{+})$ the support of $\mu$ in the s-component is bounded, the expectation is well defined. The transition law for the new model is given by $\tilde{Q}(.\vert x,\mu,z,a,b)$, which for $(x,\mu,z,a,b)\in E\times A\times B$, and a Borel measurable subset $B\in E$ is defined by

\begin{align*}
\tilde{Q}(B\vert x,\mu,z,a,b):=\int_{X}1_{B}\big((x^{'},\Phi(x,a,b,x^{'},\mu,z),\beta z)\big)Q^{X}(dx^{'}\vert x,\mu^{Y},a,b).
\end{align*}
The decision rules for player 1 in the newly defined model are given by measurable mappings $f:E\rightarrow P(A)$ such that $f(x,\mu,z)(A(x))=1$. Similarly, the decision rules for player 2 are given by measurable mappings $g:E\rightarrow P(B)$ such that $g(x,\mu,z)(B(x))=1$. We denote by $F_1$ the set of all decision rules for player 1 and $F_2$ denotes the same for player 2. For player 1, we denote by $\Pi_1^{M}$ the set of all Markov policies $\pi=(f_{0},f_{1},...)$ with $f_{n}\in F_1$ for all $n\geq 0$.  $\Pi_2^{M}$ represents the same for player 2.

Let $\mathcal{C}(E)=\{v:E\to \mathbb{R},\,v\,\, \textup{is continuous and }\,\,v\geq V_0\}$. Note that we consider the topology of weak convergence on $P_b(Y\times \mathbb{R}_{+})$. For $v\in \mathcal{C}(E)$, $(\zeta,\eta)\in P(A(x))\times P(B(x))$ and $(f,g)\in F_1\times F_2$, we consider the following operators: 
{\small\begin{align*}(T_{fg}v)(x,\mu,z):=\int_B\int_A\int_{X}v\big(x^{'},\Phi(x,a,b,x^{'},\mu,z),\beta z\big)Q^{X}(dx^{'}\vert x,\mu^{Y},a,b)f(x,\mu,z)(da)g(x,\mu,z)(db)
\end{align*}}
{\small\begin{align*}(Lv)(x,\mu,z,\zeta,\eta):=\int_B\int_A\int_{X}v\big(x^{'},\Phi(x,a,b,x^{'},\mu,z),\beta z\big)Q^{X}(dx^{'}\vert x,\mu^{Y},a,b)\zeta(da)\eta(db).
\end{align*}}

\begin{align}\label{Optimal operator}
Tv(x,\mu,z):=\inf_{g}\sup_{f}T_{fg}v=\inf_{\zeta}\sup_{\eta}(Lv)(x,\mu,z,\zeta,\eta),  \hspace{.3cm}  (x,\mu,z)\in E.
\end{align}
Next we have the following theorem:
\begin{theorem}\label{Finite horizon}
	\begin{itemize}
		\item[(a)] Let $\pi=(f_{0},f_{1},...,f_{N-1})$ and $\sigma=(g_{0},g_{1},...,g_{N-1})$ be two policies for player 1 and 2 respectively. Then it holds that for all $n=1,2,...,N$,
		$$V_{n\pi \sigma}(x,\mu,z)=T_{f_{0}g_{0}}T_{f_{1}g_{1}}...T_{f_{n-1}g_{n-1}}V_{0}.$$
		\item[(b)] For all $n=1,2,...,N$ let $V_n=\inf_{\pi}\sup_{\sigma}V_{n\pi\sigma}$. Then $V_{n}\in \mathcal{C}(E)$ and 
		\begin{align*}
		V_{n}(x,\mu,z)=TV_{n-1}(x,\mu,z).
		\end{align*}
		\item[(c)]  For $n=1,2,...,N$ there exists measurable functions $(\gamma_{n}^{*},\delta_{n}^{*})\in F_1\times F_2$ such that
		{\small\begin{align*}
		LV_{n-1}(x,\mu,z,\zeta,\delta_n^{*}(x,\mu,z)) \leq LV_{n-1}(x,\mu,z,\gamma_n^*(x,\mu,z),\delta_n^{*}(x,\mu,z)) \leq LV_{n-1}(x,\mu,z,\gamma^*_n(x,\mu,z),\eta) ,
		\end{align*}}
		for all $(\zeta,\eta)\in P(A(x))\times P(B(x))$  and $(x,\mu,z)\in E$. Then $V_{N}(\cdot,Q_0\times\delta_0,1)$ is the value of the $N$ stage partially observable stochastic game and $(\pi^{*},\sigma^{*})=(f_{n}^{*},g^{*}_{n})_{n=0,1,...,N-1}$ with $f^{*}_{n}(h_n)=\gamma_{N-n}^{*}(x_n,\mu(\cdot|h_n),\beta^n)$ and $g^{*}_{n}=\delta_{N-n}^{*}(x_n,\mu(\cdot|h_n),\beta^n)$ are optimal policies for player 1 and 2 respectively.
	\end{itemize}
	
\end{theorem}	
\begin{proof} 
(a) We establish the above iteration by induction on $n$. For $n=1$ we have,
{\footnotesize\begin{align*}
&T_{f_{0},g_{0}}V_{0}(x,\mu,z)=\int_B\int_A\int_{X} V_{0}\big(x^{'}, \Phi(x,a,b,x^{'},\mu,z),\beta z \big)Q^{X}(dx^{'}\vert x,\mu^{Y},a,b)f_0(x,\mu,z)(da)g_0(x,\mu,z)(db)\\
 &=\int_B\int_A\int_{Y}\int_{\mathbb{R}_{+}}\int_{X}\int_{\mathbb{R}_{+}}U(s^{'})\delta_{s+zC(x,y,a,b)}(ds^{'})q^{X}(x^{'}\vert x,y,a,b)\lambda(dx^{'})\mu(dy,ds)f_0(x,\mu,z)(da)g_0(x,\mu,z)(db)\\
 &=\int_{Y}\int_{\mathbb{R}_{+}}\mathbb{E}^{\pi \sigma}_{x y}\big[U\big(s+zC(x,y, A_0,B_0)\big)\big]\mu(dy,ds)\\
 &=V_{1 \pi \sigma}(x,\mu,z).
\end{align*}}
Now suppose that the statement is true for $V_{n}$. Let $\bar{\pi}=(f_{1},f_{2},...)$ and $\bar{\sigma}=(g_{1},g_{2},...)$ denote the 1-shifted policies. Then we get,
\begin{align*}
&(T_{f_{0}g_{0}}T_{f_{1}g_{1}}...T_{f_{n}g_{n}}V_{0})(x,\mu,z)\\&=\int_B\int_A\int_{X} V_{n\bar{\pi} \bar{\sigma}}\big(x^{'}, \Phi(x,a,b,x^{'},\mu,z),\beta z \big)Q^{X}(dx^{'}\vert x,\mu^{Y},a,b)f_0(x,\mu,z)(da)g_0(x,\mu,z)(db)\\
&=\int_B\int_A\int_{X}\int_{Y}\int_{\mathbb{R}_{+}}\mathbb{E}^{\bar{\pi} \bar{\sigma}}_{x^{'},y^{'}}\big[U(s^{'}+z\sum_{k=0}^{n-1}\beta^{k+1} C(X_{k},Y_{k},A_{k},B_{k}))\big]\\&\Phi(x,a,b,x^{'},\mu,z)(dy^{'},ds^{'})Q^{X}(dx^{'}\vert x,\mu^{Y},a,b)f_0(x,\mu,z)(da)g_0(x,\mu,z)(db)\\
&=\int_B\int_A\int_{X}\int_{Y}\int_{\mathbb{R}_{+}}\mathbb{E}^{\pi,\sigma}\big[U(s^{'}+z\sum_{k=1}^{n}\beta^{k} C(X_{k},Y_{k},A_{k},B_{k}))\vert X_{1}=x^{'},Y_{1}=y^{'}\big]\\&\cdot\int_{Y}\int_{\mathbb{R}_{+}}q(x^{'},y^{'}\vert x,y,a,b)\delta_{s+zC(x,y,a,b)}(ds^{'})\mu(dy,ds)\nu(dy^{'})\lambda(dx^{'})f_0(x,\mu,z)(da)g_0(x,\mu,z)(db)\\
&=\int_B\int_A\int_{Y}\int_{Y}\int_{X}\int_{\mathbb{R}_{+}}\mathbb{E}^{\pi,\sigma}\big[U(s+zC(x,y,a,b)+z\sum_{k=1}^{n}\beta^{k} C(X_{k},Y_{k},A_{k},B_{k}))\vert X_{1}=x^{'},Y_{1}=y^{'}\big]\\&q(x^{'},y^{'}\vert x,y,a,b)\mu(dy,ds)\nu(dy^{'})\lambda(dx^{'})f_0(x,\mu,z)(da)g_0(x,\mu,z)(db)\\
&=\int_{Y}\int_{\mathbb{R}_{+}}\mathbb{E}^{\pi,\sigma}\big[U(s+z\sum_{k=0}^{n}\beta^{k} C(X_{k},Y_{k},A_{k},B_{k}))\big]\mu(dy,ds)\\
&=V_{n+1 \pi \sigma}(x,\mu,z);
\end{align*}
Hence we have the desired conclusion by induction.\\

(b and c) We show by induction on $n$ that\\
(i) $V_{n}=TV_{n-1}\in \mathcal{C}(E)$\\
(ii) $T_{\gamma_{n}\delta^{*}_{n}} T_{\gamma_{n-1}\delta^{*}_{n-1}}...T_{\gamma_{1}\delta^{*}_{1}}V_{0}\leq V_{n}$, for any measurable $\gamma_{1},\gamma_{2},...,\gamma_{n}:E\rightarrow P(A(x))$.\\
(iii) $T_{\gamma^{*}_{n}\delta_{n}} T_{\gamma^{*}_{n-1}\delta_{n-1}}...T_{\gamma_{1}^{*}\delta_{1}}V_{0}\geq V_{n}$, for any measurable $\delta_{1},\delta_{2},...,\delta_{n}:E\rightarrow P(B(x))$.\\
(iv)  $T_{\gamma^{*}_{n}\delta^{*}_{n}} T_{\gamma^{*}_{n-1}\delta^{*}_{n-1}}...T_{\gamma^{*}_{1}\delta^{*}_{1}}V_{0}= V_{n}$.\\
	
Under our assumptions $V_0 \in \mathcal{C}(E)$. For $n=1$, it follows from definition that
$$V_{1}=\inf_{g}\sup_{f}T_{fg}V_{0}=TV_{0}.$$ Under our assumptions the existence of $(\gamma_1^*,\delta_1^*)$ follows from a classical measurable selection theorem \cite{Purves73} and Fan's minimax theorem \cite{Fan53}. Now $(ii)-(iv)$ follows from the definition of $(\gamma_1^*,\delta_1^*)$. Now suppose the statement is true for $n-1$. Since $V_{n-1}\in \mathcal{C}(E)$, we again have the existence of $(\gamma^{*}_{n},\delta^{*}_{n})$. Using the induction hypothesis and monotonicity of the $T$ we obtain $$T_{\gamma_{n}\delta^{*}_{n}} T_{\gamma_{n-1}\delta^{*}_{n-1}}...T_{\gamma_{1}\delta^{*}_{1}}V_{0}\leq T_{\gamma_{n} \delta_{n}^{*}}V_{n-1}\leq T_{\gamma^{*}_{n} \delta_{n}^{*}}V_{n-1}=T_{\gamma^{*}_{n}\delta^{*}_{n}} T_{\gamma^{*}_{n-1}\delta^{*}_{n-1}}...T_{\gamma^{*}_{1}\delta^{*}_{1}}V_{0}$$
for any measurable $\gamma_{1},\gamma_{2},...,\gamma_{n}:E\rightarrow P(A(x))$. On the other hand 
$$T_{\gamma^{*}_{n} \delta_{n}^{*}}V_{n-1}\leq T_{\gamma^{*}_{n} \delta_{n}}V_{n-1}\leq T_{\gamma^{*}_{n}\delta_{n}} T_{\gamma^{*}_{n-1}\delta_{n-1}}...T_{\gamma_{1}^{*}\delta_{1}}V_{0}$$
for any measurable $\delta_{1},\delta_{2},...,\delta_{n}:E\rightarrow P(B(x))$. Thus combining with part (a), we have shown that $((\gamma_N^*,\gamma_{N-1}^*,\ldots,\gamma^*_1),(\delta_N^*,\delta_{N-1}^*,\ldots,\delta^*_1))$ is a saddle point equilibrium for the $N$ stage completely observable game problem with optimization criterion given by \eqref{complete}. The rest of the conclusions now follow from the relation between the partially observable game and completely observable game.
\end{proof}
\section{Infinite Horizon Problem}	
In case of finite horizon problem we have the existence of the value of the game and also we have shown the existence of the optimal strategies. Now we consider the case of infinite horizon, i.e for a given pair of policies $(\pi,\sigma)$ and $x\in X$ we are interested in the optimization problem:
$$J_{\infty\pi\sigma}(x):=\int_{Y}\mathbb{E}^{\pi\sigma}_{xy}\big[U\big(\sum_{k=0}^{\infty}\beta^{k}C(X_{k},Y_{k},A_{k},B_{k})\big)\big]Q_{0}(dy),   \hspace{.3cm} x\in E.$$ The upper value, lower value, optimal strategies and value of the game have the same definitions as in Definition \ref{optimal} with $N$ replaced by $\infty$.
In the infinite horizon case we need to deal with concave and convex utility functions separately.\\
\textbf{Concave utility function:} We first investigate the case of concave utility function. Just as in the finite horizon case we will consider the equivalent completely observable game. To that end we have the following notations.
\begin{align}
&V_{\infty \pi \sigma}(x,\mu,z):=\int_{Y}\int_{\mathbb{R}_{+}}\mathbb{E}^{\pi \sigma}_{x y}\big[U\big( s+z\sum_{k=0}^{\infty}\beta^{k}C(X_{k},Y_{k},A_{k},B_{k})\big)\big]\mu(dy,ds),\nonumber\\&
V_{\infty}(x,\mu,z):=\inf_{\sigma}\sup_{\pi }V_{\infty \pi \sigma}(x,\mu,z).
\end{align}	
We denote 
\begin{align*}
&\bar{a}(\mu, z):=\int_{\mathbb{R}_{+}}U(s+\frac{z \bar{c}}{1-\beta})\mu^{s}(ds),\\
&\underline{a}(\mu,z):=\int_{\mathbb{R}_{+}}U(s+\frac{z\underline{c}}{1-\beta})\mu^{s}(ds), \hspace{.3cm} (\mu,z)\in P_b(Y\times \mathbb{R}_{+})\times (0,1].
\end{align*}
Then we have the main theorem of this section:
\begin{theorem}\label{concave case}
\begin{itemize}
	\item[(a)] 	$V_{\infty}$ is the unique solution of $v=Tv$ in $\mathcal{C}(E)$ with $\underline{a}(\mu,z)\leq v(x,\mu,z) \leq \bar{a}(\mu,z)$ where $T$ is as defined in (\ref{Optimal operator}). Moreover, $T^{n}V_{0}\uparrow V_{\infty}$, $T^{n}\underline{a}\uparrow V_{\infty}$ and $T^{n}\bar{a}\downarrow V_{\infty}$ as $n\rightarrow \infty$. \\
	\item[(b)] There exist measurable functions $(\gamma^{*},\delta^{*})\in F_1\times F_2$ such that
\begin{align}
T_{\gamma \delta^*}(x,\mu,z)\leq T_{\gamma^* \delta^*}(x,\mu,z)\leq T_{\gamma^* \delta}(x,\mu,z) ,
\end{align}	for all $(\gamma,\delta)\in F_1\times F_2$ and $(x,\mu,z)\in E$.
	 Then $V_{\infty}(x,Q_0\times \delta_0,1)$ is the value of the partially observable infinite horizon stochastic game. Moreover, $(\pi^{*},\sigma^{*}) = ((f_0^*,f_1^*,\ldots),(g_0^*,g_1^*,\ldots))$ with $f^{*}_{n}(h_n):= \gamma^*(x_n,\mu_n(\cdot|h_n),\beta^n)$ and \\$g^{*}_{n}(h_n):= \delta^*(x_n,\mu_n(\cdot|h_n),\beta^n)$ are optimal policies for player 1 and 2 respectively.
\end{itemize}
\end{theorem}

\begin{proof}
	(a) Here we first show that $V_{n}=T^{n}V_{0}\uparrow V_{\infty}$. Since $U:\mathbb{R}_{+}\rightarrow \mathbb{R}$ is increasing and concave, it satisfies the inequality 
$$U(s_{1}+s_{2})\leq U(s_{1})+U^{'}_{-}(s_{1})s_{2},    \hspace{.3cm} s_{1},s_{2}\geq 0,$$ 
where $U^{'}_{-}$ is the left hand side derivative of $U$ that exists since $U$ is concave. Further more, $U^{'}_{-}(s)\geq 0$ and $U^{'}_{-}$ is non increasing. For $(x,\mu,z)\in E$ it holds, $V_{n}(x,\mu,z)\leq V_{\infty}(x,\mu,z)$. Using this, we get
\begin{align}\label{ineq1}
&V_{n \pi \sigma}(x,\mu,z)\leq V_{\infty \pi \sigma}(x,\mu,z)\nonumber\\
&=\int_{Y}\int_{\mathbb{R}_{+}}\mathbb{E}^{\pi \sigma}_{x y}\big[U\big( s+z\sum_{k=0}^{\infty}\beta^{k}C(X_{k},Y_{k},A_{k},B_{k})\big)\big]\mu(dy,ds),\nonumber\\&
\leq \int_{Y}\int_{\mathbb{R}_{+}}\mathbb{E}^{\pi \sigma}_{x y}\big[U\big( s+z\sum_{k=0}^{n-1}\beta^{k}C(X_{k},Y_{k},A_{k},B_{k})\big)\big]\mu(dy,ds),\nonumber\\&
+\int_{Y}\int_{\mathbb{R}_{+}}\mathbb{E}^{\pi \sigma}_{x y}\big[U^{'}_{-}\big( s+z\sum_{k=0}^{n-1}\beta^{k}C(X_{k},Y_{k},A_{k},B_{k})\big)z\sum_{m=n}^{\infty}\beta^{k}C(X_{m},Y_{m},A_{m},B_{m})\big]\mu(dy,ds),\nonumber\\&
\leq V_{n \pi \sigma}(x,\mu,z)+\beta^{n}\frac{z\bar{c}}{1-\beta}\int_{\mathbb{R}_{+}}U^{'}_{-}(s+z\underline{c})\mu^{s}(ds)\nonumber\\&
\leq V_{n \pi \sigma}(x,\mu,z)+\beta^{n}\frac{z\bar{c}}{1-\beta}U^{'}_{-}(z\underline{c})=:V_{n \pi \sigma}(x,\mu,z)+\epsilon_{n}(z),
\end{align}
where $\epsilon_{n}(z)$ is defined by the last equation. Clearly, $\lim_{n\rightarrow \infty}\epsilon_{n}(z)=0$. Now from \eqref{ineq1} we get $V_n(x,\mu,z)\leq V_{\infty}(x,\mu,z)\leq V_n(x,\mu,z)+\epsilon_n(z)$. Taking limit we get $V_{n}=T^{n}V_{0}\uparrow V_{\infty}$. Now, $V_{n+1}=TV_n\leq TV_{\infty}$. Taking limit $n\to \infty$ we get $V_{\infty}\leq TV_{\infty}$. Again, $V_{\infty}\leq V_{n}+\epsilon_{n}$. Now applying operator $T$ on both sides we have $TV_{\infty}\leq T(V_{n}+\epsilon_{n})=V_{n+1}+\epsilon_{n+1}$. Thus again letting $n\rightarrow \infty$ we obtain $TV_{\infty}\leq V_{\infty}$. Hence combining two inequalities we have $V_{\infty}=TV_{\infty}$.\\
Next, we obtain
\begin{align*}
(T\bar{a})(\mu,z)&=\inf_{\eta}\sup_{\zeta}\int_B\int_A\int_{\mathbb{R}_{+}}U(s^{'}+\frac{z\beta \bar{c}}{1-\beta})\Phi(x,a,b,x^{'},\mu,z)(ds^{'})\zeta(da)\eta(db)\\
&\leq \int_{\mathbb{R}_{+}}U(s+z\bar{c}+\frac{z\beta \bar{c}}{1-\beta})\mu^{s}(ds)\\
&=\int_{\mathbb{R}_{+}}U(s+\frac{z \bar{c}}{1-\beta})\mu^{s}(ds)=\bar{a}(\mu,z).
\end{align*}	
By the similar reasoning it can be shown that $T\underline{a}\geq \underline{a}$. Thus we have $T^{n}\bar{a}\downarrow$ and $T^{n}\underline{a}\uparrow$ and the limits exist. Moreover, by iteration we have, 
\begin{align*}
(T^n\underline{a})(x,\mu,z)&=\inf_{\sigma}\sup_{\pi}\int_{Y}\int_{\mathbb{R}_{+}}\mathbb{E}^{\pi \sigma}_{x y}\big[U\big(s+\frac{z\beta^{n} \underline{c}}{1-\beta}+z\sum_{k=0}^{n-1}\beta^{k}C(X_{k},Y_{k},A_{k},B_{k})\big)\big]\mu(dy,ds)\\
&\geq (T^{n}V_{0})(x,\mu,z).\\
(T^n\bar{a})(x,\mu,z)&=\inf_{\sigma}\sup_{\pi}\int_{Y}\int_{\mathbb{R}_{+}}\mathbb{E}^{\pi \sigma}_{x y}\big[U\big(s+\frac{z\beta^{n} \bar{c}}{1-\beta}+z\sum_{k=0}^{n-1}\beta^{k}C(X_{k},Y_{k},A_{k},B_{k})\big)\big]\mu(dy,ds)
\end{align*} 	
Using the fact $U(s_{1}+s_{2})\leq U(s_{1})+U^{'}_{-}(s_{1})s_{2}$, we obtain
\begin{align*}
0&\leq (T^{n}\bar{a})(x,\mu,z)-(T^{n}\underline{a})(x,\mu,z) \leq (T^{n}\bar{a})(x,\mu,z)-(T^{n}V_{0})(x,\mu,z)\\
&\leq \sup_{\sigma}\sup_{\pi}\int_{Y}\int_{\mathbb{R}_{+}}\mathbb{E}^{\pi \sigma}_{x y}\big[U\big(s+\frac{z\beta^{n} \bar{c}}{1-\beta}+z\sum_{k=0}^{n-1}\beta^{k}C(X_{k},Y_{k},A_{k},B_{k})\big)\\&-U\big(s+z\sum_{k=0}^{n-1}\beta^{k}C(X_{k},Y_{k},A_{k},B_{k})\big)\big]\mu(dy,ds)\\
&\leq \epsilon_{n}(z).
\end{align*}
Since $\lim_{n\rightarrow \infty}\epsilon_{n}=0$, we obtain $T^{n}\underline{a}\uparrow V_{\infty}$ and $T^{n}\bar{a}\downarrow V_{\infty}$ as $n\rightarrow \infty$. Also since $V_{\infty}$ is both increasing as well as decreasing limit of continuous functions, it is also continuous.\\
Now for the uniqueness purpose if possible let there be another solution $v\in \mathcal{C}(E)$	of $v=Tv$ with $\underline{a}\leq v\leq \bar{a}$. This then implies that $T^{n}\underline{a}\leq v\leq T^{n}\bar{a}$ for all $n$. Taking the limit $n\rightarrow \infty$ we have the uniqueness of the solution.\\
(b) The existence of $(\gamma^{*},\delta^{*})$ follows again from the measurable selection theorem and the minimax theorem. By monotonicity and the
fact that $V_0\leq V_{\infty} \leq \bar{a}$ we obtain that $\lim_{n\rightarrow \infty}T^{n}_{\gamma^{*},\delta^{*}}V_{0}=\lim_{n\rightarrow \infty}T^{n}_{\gamma^{*},\delta^{*}}V_{\infty}=V_{\infty \pi_{\gamma^*}\sigma_{\delta^*}}$, where $\pi_{\gamma^*}=(\gamma^*,\gamma^*,\ldots)$ and $\sigma_{\delta^*}=(\delta^*,\delta^*,\ldots)$. By the definition of $(\gamma^{*},\delta^{*})$ we obtain
for any $(\gamma,\delta)\in F_1\times F_2$, 
$$T_{\gamma\delta^{*}}V_{\infty} \leq T_{\gamma^{*}\delta^{*}}V_{\infty} \leq T_{\gamma^{*}\delta}V_{\infty}.$$ The property of $(\gamma^{*},\delta^{*})$ also implies that $V_{\infty}=TV_{\infty}=T_{\gamma^{*}\delta^{*}}V_{\infty}$.
 Hence we can also write for any $(\gamma,\delta)\in F_1\times F_2$, 
$$T_{\gamma\delta^{*}}V_{\infty} \leq V_{\infty} \leq T_{\gamma^{*}\delta}V_{\infty}.$$
By iterating this inequality $n$-times we get
$$T_{\gamma_{1}\delta^{*}}T_{\gamma_{2}\delta^{*}}...T_{\gamma_{n}\delta^{*}}V_{\infty} \leq V_{\infty} \leq T_{\gamma^{*}\delta_{1}}T_{\gamma^{*}\delta_{2}}...T_{\gamma^{*}\delta_{n}}$$ for arbitrary $\gamma_{1},\gamma_{2},...,\gamma_{n}$ and $\delta_{1},\delta_{2},...,\delta_{n}$. Letting $n\rightarrow \infty$ we get,
$$V_{\infty \pi \sigma_{\delta^*} }\leq V_{\infty \pi_{\gamma^*} \sigma_{\delta^*} }=V_{\infty}\leq V_{\infty \pi_{\gamma^*} \sigma }$$
for all policies $\pi$ and $\sigma$. The rest of the conclusions are now straight forward.

\end{proof}
\noindent\textbf{Convex Utility function:} Now we consider the case of convex utility function $U$. 
\begin{theorem}
	Theorem \ref{concave case} also holds for convex $U$.
\end{theorem}
\begin{proof} For the convex case the proof is along the same lines as in Theorem \ref{concave case}. The only difference is that here we will use another inequality. Note that for $U:\mathbb{R}_{+}\rightarrow \mathbb{R}$ increasing and convex we have the inequality
$$U(s_{1}+s_{2})\leq U(s_{1})+U^{'}_{+}(s_{1}+s_{2})s_{2},  \hspace{.3cm} s_{1},s_{2}\geq 0,$$
where $U^{'}_{+}$ is the right-hand side derivative of $U$ that exists since $U$ is convex. Moreover, $U^{'}_{+}(s)\geq 0$ and $U^{'}_{+}$ is increasing. Thus, we obtain for $(x,\mu,z)\in E$,
\begin{align*}
&V_{n \pi \sigma}(x,\mu,z)\leq V_{\infty \pi \sigma}(x,\mu,z)\nonumber\\
&=\int_{Y}\int_{\mathbb{R}_{+}}\mathbb{E}^{\pi \sigma}_{x y}\big[U\big( s+z\sum_{k=0}^{\infty}\beta^{k}C(X_{k},Y_{k},A_{k},B_{k})\big)\big]\mu(dy,ds),\nonumber\\&
\leq \int_{Y}\int_{\mathbb{R}_{+}}\mathbb{E}^{\pi \sigma}_{x y}\big[U\big( s+z\sum_{k=0}^{n-1}\beta^{k}C(X_{k},Y_{k},A_{k},B_{k})\big)\big]\mu(dy,ds),\nonumber\\&
+\int_{Y}\int_{\mathbb{R}_{+}}\mathbb{E}^{\pi \sigma}_{x y}\big[U^{'}_{+}\big( s+z\sum_{k=0}^{\infty}\beta^{k}C(X_{k},Y_{k},A_{k},B_{k})\big)z\sum_{m=n}^{\infty}\beta^{k}C(X_{m},Y_{m},A_{m},B_{m})\big]\mu(dy,ds),\nonumber\\&
\leq V_{n \pi \sigma}(x,\mu,z)+\beta^{n}\frac{z\bar{c}}{1-\beta}\int_{\mathbb{R}_{+}}U^{'}_{+}(s+\frac{z\bar{c}}{1-\beta})\mu^{s}(ds)\nonumber\\&
\end{align*}
Now let's denote $\delta_{n}(\mu,z):=\beta^{n}\frac{z\bar{c}}{1-\beta}\int_{\mathbb{R}_{+}}U^{'}_{+}(s+\frac{z\bar{c}}{1-\beta})\mu^{s}(ds)$. Then we have $\lim_{n\rightarrow \infty}\delta_{n}(\mu,z)=0$. So we end up with 
$$V_{n}(x,\mu,z)\leq V_{\infty}(x,\mu,z)\leq V_{n}(x,\mu,z)+\delta_{n}(\mu,z).$$
Letting $n\rightarrow \infty$ yields $T^{n}V_{0}\rightarrow V_{\infty}$.\\
We use the same inequality to get,
\begin{align*}
0&\leq (T^{n}\bar{a})(x,\mu,z)-(T^{n}\underline{a})(x,\mu,z) \leq (T^{n}\bar{a})(x,\mu,z)-(T^{n}V_{0})(x,\mu,z)\\
&\leq \sup_{\sigma}\sup_{\pi}\int_{Y}\int_{\mathbb{R}_{+}}\mathbb{E}^{\pi \sigma}_{x y}\big[U\big(s+\frac{z\beta^{n} \bar{c}}{1-\beta}+z\sum_{k=0}^{n-1}\beta^{k}C(X_{k},Y_{k},A_{k},B_{k})\big)\\&-U\big(s+z\sum_{k=0}^{n-1}\beta^{k}C(X_{k},Y_{k},A_{k},B_{k})\big)\big]\mu(dy,ds)\\
&\leq \beta^{n}\frac{z\bar{c}}{1-\beta}\int_{\mathbb{R}_{+}}U^{'}_{+}(s+\frac{z\bar{c}}{1-\beta})\mu^{s}(ds)=\delta_{n}(\mu,z)
\end{align*}
and the right hand side converges to 0 as $n\rightarrow \infty$.
\end{proof}

Few remarks are in order.
\begin{remark}
(1) An important sub case of the model that we have considered here is when the reward/cost does not depend on the unobservable component, i.e., $C(x,y,a,b)=C^{\prime}(x,a,b)$ for some function $C^{\prime}(\cdot)$. In this case the accumulated reward/cost is no more unobservable and thereby we need not estimate it. Thus in that case it can be shown along similar lines as in Proposition 1 of \cite{Rieder17a}, that we can take the state space of the completely observable model as $X\times P(Y)\times \mathbb{R}_+\times (0,1]$ and the updating operator as $$\Phi(x,a,b,x^{'},\mu, z)(B):=\frac{\int_{Y}\big(\int_{B}q(x^{'},y^{'}\vert x,y,a,b)\nu(dy^{'})\big)\mu(dy)}{\int_{Y}q^{X}(x^{'}\vert x,y,a,b)\mu(dy)},$$ where $B$ is a Borel subset of $Y$ and $\mu \in P(Y)$.

(2) An important utility function is given by the function $U(x)=\frac{1}{\theta}e^{\theta x}$, where $\theta > 0$ is a fixed parameter. In this case again it is not necessary to keep track of the accumulated cost. It is enough to consider $X\times P(Y)\times (0,1]$ as the state space of the completely observable model. Again arguing similar to Theorem 3 in \cite{Rieder17a}, it can be shown that the value of the $N$ stage game problem is given by $J(x)=\alpha_N(x,Q_0,\theta)$ where $\alpha_0(x,\mu,\theta z)=\frac{1}{\theta}$ and for $n=1,2,\ldots,N$,
{\footnotesize\begin{align*}
\alpha_{n+1}(x,\mu,\theta z)=\inf_{\eta \in P(B(x))}\sup_{\zeta \in P(A(x))}\int_{B}\int_A\left[\alpha_n(x^{\prime},\Phi_e(x,a,b,x^{\prime},\mu,z),\beta \theta z)\hat{Q}^X(dx^{\prime}|x,\mu,a,b,\theta z)\right]\zeta(da)\eta(db), 
\end{align*} }
where $(x,\mu,z)\in X\times P(Y)\times (0,1]$ and for $B_1$, Borel subset of $X$ and $B_2$, Borel subset of $Y$,
$$ \hat{Q}^X(B_1|x,\mu,a,b,z)=\int_{B_1}\int_Y e^{zC(x,y,a,b)}q^X(x^{\prime}|x,y,a,b)\mu(dy)\lambda(dx^{\prime}),$$
$$\Phi_e(x,a,b,x^{\prime},\mu,z)(B_2)=\frac{\int_{B_2}\int_Y e^{zC(x,y,a,b)}q(x^{\prime},y^{\prime}|x,y,a,b)\mu(dy)\nu(dy^{\prime})}{\int_{Y}\int_Y e^{zC(x,y,a,b)}q(x^{\prime},y^{\prime}|x,y,a,b)\mu(dy)\nu(dy^{\prime})}.$$

\end{remark}
	
\bibliographystyle{plain}
	\bibliography{bib}		
	
\end{document}